\documentclass[11pt,a4paper]{amsart}
\usepackage{amsfonts,amsmath,amssymb,amsthm}
\usepackage{times}
\usepackage{graphics,graphicx}
\usepackage{color}
\usepackage{enumerate}

\usepackage[utf8]{inputenc}
\usepackage{bbm, dsfont}

\numberwithin{equation}{section}

\renewcommand{\epsilon}{\varepsilon}

\DeclareSymbolFont{SY}{U}{psy}{m}{n}
\DeclareMathSymbol{\emptyset}{\mathord}{SY}{'306}

\DeclareMathOperator{\const}{const}

{\bf}{\it}

\newtheorem{theorem}{Theorem}[section]{\bf}{\it}
{\bf}{\it}
{\bf}{\it}
{\bf}{\it}
{\it}{\rm}
\newtheorem{lemma}[theorem]{Lemma}{\bf}{\it}
\newtheorem{remark}[theorem]{Remark}{\it}{\rm}
{\bf}{\it}
{\bf}{\it}
{\bf}{\it}

{\bf}{\it}

{\bf}{\it}

\title[Asymptotic degree distribution of a duplication--deletion random graph model]{Asymptotic degree distribution of a duplication--deletion random graph model}


\author[E.~Th\"ornblad]{Erik Th\"ornblad \\ Department of Mathematics, Uppsala University}
 \address{E.~Th\"ornblad, Department of Mathematics, University of Uppsala, Box 480, S-75106 Uppsala, Sweden}
 \email{erik.thornblad@math.uu.se}

\date{\today}

\begin{document}

\begin{abstract}
 We study a discrete--time duplication--deletion random graph model and analyse its asymptotic degree distribution. The random graphs consists of disjoint cliques. In each time step either a new vertex is brought in with probability $0<p<1$ and attached to an existing clique, chosen with probability proportional to the clique size, or all the edges of a random vertex are deleted with probability $1-p$. We prove almost sure convergence of the asymptotic degree distribution and find its exact values in terms of a hypergeometric integral, expressed in terms of the parameter $p$. In the regime $0<p<\frac{1}{2}$ we show that the degree sequence decays exponentially at rate $\frac{p}{1-p}$, whereas it satisfies a power--law with exponent $\frac{p}{2p-1}$ if $\frac{1}{2}<p<1$. At the threshold $p=\frac{1}{2}$ the degree sequence lies between a power--law and exponential decay.
\end{abstract}
\maketitle

\section{Introduction}
Over the past few decades dynamic random graph models have been studied extensively. These are random graphs which evolve over time, for instance by introducing new vertices and edges (and possibly deleting them). Perhaps the most well--known example of a dynamic random graph is the preferential attachment model introduced by Barab\'asi and Albert \cite{BarabasiAlbert}. In this model new vertices are brought in, one at a time, and are connected to a single old vertex chosen with probability proportional to the degree of the target vertex. Subsequently this was analysed in among others \cite{BollobasRiordanScaleFree,MoriRandomTrees}. One of the main results on the Barab\'asi--Albert preferential attachment model is that the degree distribution satisfies a power--law almost surely, that is, that the asymptotic proportion of vertices of degree $k$ decays like $k^{-\alpha}$ for some constant $\alpha>0$. Many other models incorporating some sort of preferential attachment rule also exhibit power--law behaviour. Furthermore, it is possible to allow for more complicated dynamics without destroying the power--law property. One possible variation is to also allow for deletion of vertices or edges. Chung and Lu \cite{ChungLu2004} considered a dynamic random graph model that allowed for addition of new vertices and edges (with endpoints chosen proportionally to degree) and deletion of randomly chosen vertices or edges. They determined a number of properties of this random graph. In particular they showed that the degree distribution satisfied a power law almost surely, with the coefficient being a function of the addition and deletion probabilities. 

However, in several models allowing for edge or vertex deletion, a phase transition is seen for the asymptotic degree distribution. Typically this phase transition occurs when the probability of deletion is too high. For instance, Deijfen and Lindholm \cite{DeijfenLindholm} considered a model for which the asymptotic expected degree distribution satisfied a power--law behaviour in the regime below some threshold deletion probability, while it decayed  exponentially above this threshold value. They did however not comment on the behaviour at the threshold value. Wu, Dong, Liu and Cai \cite{WuDongLiuCai} achieved similar results and also found the asymptotic degree distribution at the threshold probability for another random graph model. The work in \cite{DeijfenLindholm, WuDongLiuCai} was subsequently extended by Vallier \cite{Vallier}, who showed that a similar phenomenon occurs in a random graph model considered by Cooper, Frieze and Vera, who in their original paper \cite{CooperFriezeVera} had restricted themselves to the power--law regime. We shall see a similar phase transition for the random graph model considered here.

In this paper we study the following random graph process. Let $0<p<1$ be fixed. Start at time $m=0$ with a graph $G_0$ consisting of a single isolated vertex. At integer times $m\geq 1$ generate $G_m$ from $G_{m-1}$ by doing one of the following steps.
\begin{enumerate}[i.]
 \item With probability $p$, do a \emph{duplication step}. Choose uniformly at random a vertex $v$ from $G_{m-1}$. Introduce a new vertex $w$ and form edges between $w$ and all the neighbours of $v$. Also form an edge between $v$ and $w$.
 \item With probability $1-p$, do a \emph{deletion step}. Choose uniformly at random a vertex from $G_{m-1}$ and make it isolated by deleting all its incident edges. 
\end{enumerate}

\begin{figure}[ht]
\def\svgwidth{400pt}
 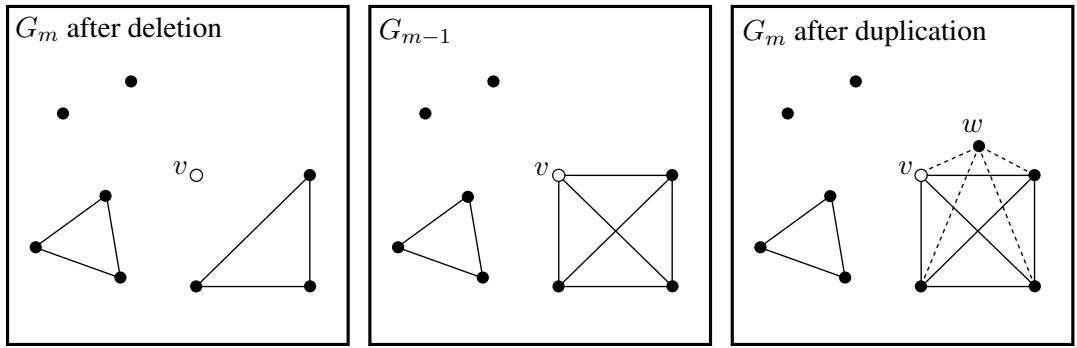
\label{figure}
\caption{A schematic of a single step. The middle figure shows the graph at some time $m-1$. The white vertex $v$ is the vertex chosen for duplication or deletion (this choice is done uniformly at random among all vertices). The right--most graph is the resulting graph if duplication occurs, which happens with probability $p$. The left--most graph is the resulting graph after a deletion step, which happens with probability $1-p$.} 
\end{figure}

The results in this paper concern the limiting degree distribution of the vertices in the graph, that is, the limiting proportion of vertices of degree $k$. We will refer to the three regimes $0<p<\frac{1}{2}$, $p=\frac{1}{2}$ and $\frac{1}{2}<p<1$ as the subcritical case, the critical case and the supercritical case respectively. Since duplication and deletion steps affect a whole clique, vertices in large cliques are more likely to have edges added and edges removed. In this sense edges are attached and deleted preferentially. The subcritical, critical and supercritical cases will require separate analysis, and we will see that a phase transition in the asymptotic degree distribution occurs, passing from exponential decay to power--law at the critical value $p=\frac{1}{2}$. 

This model is motivated by a recent paper \cite{BackhauszMori} of Backhausz and M\'ori. The random graph process introduced by them has the same deletion and duplication steps, but duplication and deletion steps are carried out alternatingly. Hence there is no parameter $p$ and no randomness in the step type. We note that we in the critical case $p=\frac{1}{2}$ on average do as many duplication steps as deletion steps. One might therefore expect that the model considered by Backhausz and M\'ori has similar properties to our model in the critical case. We will see later that the asymptotic degree distribution in the critical case agrees with the one found by Backhausz and M\'ori for their model.

In a biological context it is more natural to study the corresponding continuous--time version of the model. Such studies were done recently in \cite{ChampagnatLambert2012, ChampagnatLambert2012b} and earlier in \cite{pakes1989}. In this framework one would attach to each clique two exponential clocks, ringing at rate $k\lambda$ and $k\mu$ respectively, where $k$ is the size of the clique and $\lambda, \mu >0$ two parameters. If the first clock rings, a new vertex is added to the clique. If the second clock rings, a vertex is removed from the clique and made isolated. Indeed, the findings in the present paper agree with the results given in \cite{ChampagnatLambert2012} in the critical and supercritical cases. They do however not give a corresponding result for the subcritical case.

The rest of the paper is outlined as follows. Let $D_{m,k}$ denote the number of vertices of degree $k-1$ (that is, the number of vertices in cliques of size $k$) at time $m$, and let $N_m$ denote the number of vertices at time $m$. In Section 2 we prove that there exists a unique positive bounded sequence $(d_k)_{k=1}^{\infty}$ such that $\liminf_{m\to \infty}D_{m,k}/N_m\geq d_k$. In Section 3 we prove that the limit $\lim_{m\to \infty}D_{m,k}/N_m = d_k$ exists almost surely. The critical case was analysed already in \cite{BackhauszMori}, but in Section $4$ we consider the limit sequence $(d_k)_{k=1}^{\infty}$ for the subcritical and supercritical cases. Using Laplace's method we derive its exact values expressed in terms of certain hypergeometric integrals. We also consider the asymptotics of these integrals and show that the supercritical case gives rise to a power law with exponent $\beta=\frac{p}{2p-1}$. For the subcritical case we obtain that $d_k$ decays exponentially like $\gamma^{-k}$ where $\gamma=\frac{1-p}{p}$. Finally we show that this sequence indeed defines a probability distribution for all $0<p<1$.

\section{A bound on the limit inferior}

This section is devoted to proving the following theorem. 
\begin{theorem}
 Let $D_{m,k}$ denote the number of vertices of degree $k-1$ (that is, the number of vertices in $k$--cliques) at time $m$. Denote by $N_m$ the total number of vertices at time $m$. Then for each $k\geq 1$ we have that
\begin{align}
\liminf_{m\to \infty}\frac{D_{m,k}}{N_m} \geq d_k
\end{align}
almost surely, where $(d_k)_{k=0}^{\infty}$ is the unique positive bounded sequence satisfying
\begin{align}
d_0&=\frac{1-p}{p},  \\
d_k&=\frac{pkd_{k-1}+(1-p)kd_{k+1}}{k+p}, \quad k\geq 1. 
\end{align}
\label{thm:as:degree0} 
\end{theorem}
It will later be shown that the sequence $(d_k)_{k=1}^{\infty}$ indeed is the asymptotic degree distribution. We will prove Theorem \ref{thm:as:degree0} by proving a similar result on the clique sizes. First we find expressions for the expected number of $k$--cliques at time $m$, conditional on the graph at time $m-1$. Let $C_{m,k}$ denote the number of $k$--cliques at time $m$. Denote by $\mathcal{F}_m$ the $\sigma$--field generated by the random graph process up to and including time $m$.  Since we start with a single isolated vertex at time $0$ we have the boundary conditions $C_{0,1}=1$ and $C_{0,k}=0$ for $k\geq 2$. 

The $1$--cliques at time $m$ originate from three sources.
\begin{enumerate}
 \item A $1$--clique at time $m-1$ not chosen for duplication (recall that a $1$--clique selected for deletion is left unaffected). This happens for each $1$--clique with probability $1-p/N_{m-1}$.
 \item A $k$--clique ($k>2$) was selected for deletion. This happens with probability $(1-p)k/N_{m-1}$ for each $k$--clique and gives rise to one new $1$--clique.
 \item A $2$--clique was selected for deletion. This happens with probability $(1-p)2/N_{m-1}$ for each $2$--clique, and this gives rise to two new $1$--cliques.
\end{enumerate}
This leads to the following conditional expectation of the number of $1$--cliques.
\begin{align}
\mathbb{E}[C_{m,1} \ | \ \mathcal{F}_{m-1}]  
& \ =\  C_{m-1,1}\left(1-p\frac{1}{N_{m-1}} \right) \\ 
& \qquad + \ 4(1-p)  \frac{C_{m-1,2}}{N_{m-1}} \notag \\
& \qquad + \ (1-p)\sum_{i=3}^{m}  \frac{iC_{m-1,i}}{N_{m-1}}. \notag
\end{align}

Note that $\sum_{i=3}^{m}iC_{m-1,i}=N_{m-1}-C_{m-1,1}-2C_{m-1,2}$. Thus 
\begin{align}
\mathbb{E}[C_{m,1} \ | \ \mathcal{F}_m ] 
& \ = \ C_{m-1,1}\left(1-\frac{1}{N_{m-1}} \right)+(1-p)+2(1-p) \frac{C_{m-1,2}}{N_{m-1}}.
\label{exp:1}
\end{align}

We do a similar analysis for the number of $k$--cliques ($k\geq 2$) at time $m$. These come from three sources.
\begin{enumerate}
 \item The $k$--cliques at time $m-1$ which were not selected for duplication nor deletion. For each $k$--clique this happens with probability $1-k/N_{m-1}$.
 \item A $(k-1)$--clique at time $m-1$ was selected for duplication. This happens with probability $p(k-1)/N_{m-1}$ and gives rise to one new $k$--clique.
 \item A $(k+1)$--clique at time $m-1$ was selected for deletion. This happens with probability $(1-p)(k+1)/N_{m-1}$ and gives rise to one new $k$--clique. 
\end{enumerate}
This gives us the following conditional expectation of the number of $k$--cliques at time $m$:
\begin{align}
\mathbb{E}[C_{m,k} \ | \ \mathcal{F}_{m-1}]  \label{exp:k}
& \ =\  C_{m-1,k}\left(1-\frac{k}{N_{m-1}} \right) \\ 
& \qquad + \ p(k-1) \frac{C_{m-1,k-1}}{N_{m-1}} \notag \\
& \qquad +\ (1-p)(k+1)\frac{C_{m-1,k+1}}{N_{m-1}}. \notag
\end{align}

We now quote an essential lemma due to Backhausz and M\'ori \cite{BackhauszMoriLemma}. The proof uses martingale techniques. The following version is a slightly less general version adapted to our purposes. There is a corresponding result for the limit superior which we shall not use, and therefore do not quote here.
\begin{lemma}[Backhausz and M\'ori \cite{BackhauszMoriLemma}]
 Let $(\mathcal{F}_m)_{m=0}^{\infty}$ be a filtration. Let $(\xi_m)_{m=0}^{\infty}$ be a non--negative process adapted to $(\mathcal{F}_m)_{m=0}^{\infty}$, and let $(u_m)_{m=1}^{\infty},(v_m)_{m=1}^{\infty}$ be non--negative predictable processes such that $u_m<m$ for all $m\geq 1$ and $\lim_{m\to \infty} u_m = u>0$ exists almost surely. Let $w$ be a positive constant. Suppose that there exists $\delta >0$ such that $\mathbb{E}[(\xi_m-\xi_{m-1})^2 | \mathcal{F}_m]=O(m^{1-\delta})$.

If \begin{align}
    \liminf_{m\to \infty}\frac{v_m}{w}\geq v
   \end{align} for some constant $v\geq 0$ and 
\begin{align}
 \mathbb{E}[\xi_m | \mathcal{F}_{m-1}]\geq \left(1-\frac{u_m}{m} \right)\xi_{m-1}+v_m,
\end{align}
then
\begin{align}
 \liminf_{m\to \infty}\frac{\xi_m}{mw}\geq \frac{v}{u+1} \quad \text{ a.s.}
\end{align}
\label{lem:backhauszmori}
\end{lemma}

In order to prove Theorem \ref{thm:as:degree0} we will prove the following theorem.

\begin{theorem}
  Let $C_{m,k}$ denote the number of $k$--cliques at time $m$. Denote by $N_m$ the total number of vertices in the graph at time $m$. Then for all $k\geq 1$ we have that
 \begin{align}
 \liminf_{m \to \infty}\frac{C_{m,k}}{N_m} \geq c_k
 \end{align}
 almost surely, where $(c_k)_{k=1}^{\infty}$ is a positive sequence satisfying
 \begin{align}
 c_1&=\frac{(1-p)(1+2c_2)}{1+p}, \label{def:seq1}\\ 
 c_k&=\frac{p(k-1)c_{k-1}+(1-p)(k+1)c_{k+1}}{k+p}, \quad k\geq 2. \label{def:seq2}
 \end{align}
 \label{thm:as:cliques}
\end{theorem}

We outline the idea behind the proof of Theorem \ref{thm:as:cliques}. 
\begin{enumerate}
 \item Let $A_k=\liminf_{m \to \infty}\frac{C_{m,k}}{N_m}$ for each $k\geq 1$. For each $k\geq 1$, we find a sequence $(a_k^{(j)})_{j=0}^{\infty}$ such that $a_k^{(j)}\leq A_k$ for all $j\geq 0$. This sequence is defined recursively and the inequalities are proved by induction.
\item For each $k\geq 1$, we prove that the sequence $(a_k^{(j)})_{j=0}^{\infty}$ is monotonically increasing (in $j$). Since it lies in the bounded set $[0,1]$ it must be convergent to some $a_k$. By construction we have that $a_k\leq A_k$ for all $k\geq 1$.
\end{enumerate}
The proof is similar to the first half of the proof of Proposition 5 in \cite{BackhauszMori} by Backhausz and M\'ori. They however also bound the limit superior from above and show that the limit superior and limit inferior are equal. This would work in this case as well (at least after one passes from the cliques to the individual vertices as we do in the next section), but we prefer another method.

\begin{proof}[Proof of Theorem \ref{thm:as:cliques}.]
\textbf{Step 1.} Define $A_k=\liminf_{m \to \infty}\frac{C_{m,k}}{N_m}$. Since the number of vertices increases by $1$ in a duplication step and is left unaffected in a deletion step, the strong law of large numbers implies that 
\begin{align}
 \lim_{m\to \infty}\frac{N_m}{pm}=1,
\end{align}
which means that
\begin{align} A_k=\liminf_{m\to \infty}\frac{C_{m,k}}{pm}.\end{align}
This representation will allow us to invoke Lemma \ref{lem:backhauszmori}, so we will use this from now on.

For each $k\geq 1$ we now construct a sequence $(a_k^{(j)})_{j=0}^{\infty}$ and prove by induction that it satisfies $a_k^{(j)}\leq A_k$ for all $j\geq 0$. Let $a_k^{(0)}=0$ for all $k\geq 1$. Define recursively the sequence $(a_k^{(j)})_{j=0}^{\infty}$ by
\begin{align}
 a_1^{(j+1)}=\frac{(1-p)(1+2a_2^{(j)})}{1+p} \label{k=1}
\end{align}
and
\begin{align}
 a_k^{(j+1)}=\frac{p(k-1)a_{k-1}^{(j)}+(1-p)(k+1)a_{k+1}^{(j)}}{k+p}, \qquad k\geq 2. \label{k=k}
\end{align}

The induction statement is true for $j=0$ by definition. Suppose that there exists some $j$ such that $a_k^{(j)}\leq A_k$ for all $k\geq 1$. For $k=1$ define the following variables, where we use the notation in Lemma \ref{lem:backhauszmori}.
\begin{align}
\begin{cases}
 \xi_m=C_{m,1}, \\ 
w=p, \\
u_m=\frac{m}{N_{m-1}}, \\ 
v_m=(1-p)+2(1-p) \frac{C_{m-1,2}}{N_{m-1}}.
\end{cases}
\end{align}
Note that $u_m$ and $v_m$ both are positive predictable sequences, that $\xi_m$ is a non--negative adapted sequence and that $u_m\to \frac{1}{p}=:u$ a.s. By the induction hypothesis $\liminf_{m \to \infty}\frac{C_{m-1,2}}{N_m-1}\geq a_2^{(j)}$. Choosing $v=\frac{(1-p)(1+2a_2^{(j)})}{p}$ we ensure that $\liminf_{m \to \infty}\frac{v_m}{w}\geq v$. The technical condition $\mathbb{E}[(\xi_m-\xi_{m-1})^2|\mathcal{F}_{m-1}]=O(m^{1-\delta})$ is satisfied since the maximum change in the number of $1$--cliques is $2$.
 By Lemma \ref{lem:backhauszmori} and (\ref{exp:1}) we have that
\begin{align}
\begin{split}
 A_1
=\liminf_{m \to \infty}\frac{C_{m,1}}{N_m} 
=\liminf_{m\to \infty}\frac{C_{m,1}}{pm}  
&\geq \frac{v}{u+1} \\
&=\frac{(1-p)(1+2a_2^{(j)})}{p(\frac{1}{p}+1)}\\
&=\frac{(1-p)(1+2a_2^{(j)})}{1+p} \\
&=a_1^{(j+1)}.
\end{split}
\end{align}

For $k>1$ we define the following:
\begin{align}
\begin{cases}
 \xi_m=C_{m,k}, \\
w=p, \\
u_m=\frac{km}{N_{m-1}}, \\ 
v_m=p(k-1) \frac{C_{m-1,k-1}}{N_{m-1}}+(1-p)(k+1) \frac{C_{m-1,k+1}}{N_{m-1}}.
\end{cases}
\end{align}

Note that $u_m\to \frac{k}{p}$ almost surely and that 
\begin{align}
 \liminf_{m \to \infty} \frac{v_m}{p}\geq \frac{p(k-1) a_{k-1}^{(j)}+(1-p)(k+1) a_{k+1}^{(j)}}{p}=:v.
\end{align}
by the induction hypothesis. The technical condition $\mathbb{E}[(\xi_m-\xi_{m-1})^2|\mathcal{F}_{m-1}]=O(m^{1-\delta})$ is satisfied. By Lemma \ref{lem:backhauszmori} and (\ref{exp:k}) we have that
\begin{align}
\begin{split}
 A_k
=\liminf_{m \to \infty}\frac{C_{m,k}}{N_m} 
=\liminf_{m\to \infty}\frac{C_{m,k}}{pm}  
&\geq \frac{v}{u+1} \\
&=\frac{p(k-1) a_{k-1}^{(j)}+(1-p)(k+1) a_{k+1}^{(j)}}{p(k/p+1)} \\
&=\frac{p(k-1) a_{k-1}^{(j)}+(1-p)(k+1) a_{k+1}^{(j)}}{k+p} \\
&=a_k^{(j)}.
\end{split}
\end{align}

This completes the proof that $a_k^{(j)}<A_k$ for all $j\geq 0$ and all $k\geq 1$.

\textbf{Step 2.} We now prove that the sequence $(a_k^{(j)})_{j=0}^{\infty}$ is increasing for any fixed $k$. We prove this by induction. Since $a_k^{(0)}=0$ for all $k$, it is clear that $a_k^{(1)}\geq a_k^{(0)}$. Suppose that the induction statement is true for some $j$.
 For $k=1$ we have that
\begin{align}
 a_1^{(j+1)}=\frac{(1-p)(1+2a_2^{(j)})}{1+p} \geq \frac{(1-p)(1+2a_2^{(j-1)})}{1+p}=a_1^{(j)}.
\end{align}
For $k>1$ we have that
\begin{align}
 a_k^{(j+1)}=\frac{p(k-1) a_{k-1}^{(j)}+(1-p)(k+1) a_{k+1}^{(j)}}{k+p}\geq \frac{p(k-1) a_{k-1}^{(j-1)}+(1-p)(k+1) a_{k+1}^{(j-1)}}{k+p}=a_k^{(j)}.
\end{align}
Hence the sequence $(a_k^{(j)})_{j=0}^{\infty}$ is an increasing sequence. Since $0\leq a_k^{(j)}\leq A_k\leq 1$ we have that the sequence is bounded above by $1$, and so it must be convergent and have a limit $a_k$. The limit sequence $(a_k)_{k=1}^{\infty}$ satisfies the recurrence
\begin{align}
 a_k=&\frac{p(k-1)a_{k-1}+(1-p)(k+1)a_{k+1}}{k+p} \qquad k\geq 2, \label{eq:reccurencek}
\end{align}
\begin{align}
a_1=&\frac{(1-p)(1+2a_2)}{1+p} \label{eq:recurrence1}
\end{align}
which is seen by taking limits in (\ref{k=1}) and (\ref{k=k}).
 \end{proof}

Let us now instead consider the degree distribution and use Theorem \ref{thm:as:cliques} to prove Theorem \ref{thm:as:degree0}, which we restate here for convenience. The first part of the proof follows straightforwardly from Theorem \ref{thm:as:cliques}.

 \begin{theorem}
  Let $D_{m,k}$ denote the number of vertices of degree $k-1$ (that is, the number of vertices in $k$--cliques) at time $m$. Denote by $N_m$ the total number of vertices at time $m$. Then for each $k\geq 1$ we have that
 \begin{align}
 \liminf_{m\to \infty}\frac{D_{m,k}}{N_m} \geq d_k
 \end{align}
 almost surely, where $(d_k)_{k=0}^{\infty}$ is the unique positive bounded sequence satisfying
 \begin{align}
 d_0&=\frac{1-p}{p}, \label{def:degreeseq0} \\
 d_k&=\frac{pkd_{k-1}+(1-p)kd_{k+1}}{k+p}, \quad k\geq 1. \label{def:degreeseq2}
 \end{align}
 \label{thm:as:degree} 
 \end{theorem}

\begin{proof}
There are $k$ vertices in each $k$--clique, so $D_{m,k}=kC_{m,k}$. By Theorem \ref{thm:as:cliques} we have that $D_{m,k}/N_m\to kc_k$ almost surely. Defining $d_0=\frac{1-p}{p}$ and putting $d_k=kc_k$ in (\ref{def:seq1}) and (\ref{def:seq2}), we retrieve (\ref{def:degreeseq2}) for $k\geq 1$.

Since $0\leq \frac{D_{m,k}}{N_m}\leq 1$ and the sequence $(d_k)_{k=1}^{\infty}$ is positive, we have that $(d_k)_{k=1}^{\infty}$ is a bounded sequence in $(0,1]$.

For uniqueness, suppose for contradiction that there are two distinct bounded sequences satisfying (\ref{def:degreeseq0}) and (\ref{def:degreeseq2}) above, $(d_k)_{k=0}^{\infty}$ and $(\hat{d_k})_{k=0}^{\infty}$ say. Without loss of generality suppose that $d_1>\hat{d_1}$. Define $\theta_k=d_k-\hat{d_k}$ for all $k\geq 0$. This is a bounded sequence satisfying the homogeneous system of equations
\begin{align}
\theta_0&=0, \label{def:hom0} \\
\theta_{k+1}&=\frac{(k+p)\theta_k-pk\theta_{k-1}}{(1-p)k}, \quad k\geq 1. \label{def:hom2}
\end{align}

We prove by induction that 
\begin{align}
 \theta_{k+1}\geq \left(1+\frac{p}{(1-p)k+\mathbbm{1}_{\{k=0\}}} \right)\theta_k
\end{align}
for all $k\geq 0$. Since $\theta_0=0$ and $\theta_1=d_1-\hat{d_1}>0$ this is certainly true for $k=0$. Suppose it holds true for $k-1$. This implies trivially that $\theta_{k}\geq \theta_{k-1}$. But then
\begin{align}
 \theta_{k+1}= \frac{(k+p)\theta_k-pk\theta_{k-1}}{(1-p)k}\geq \frac{k+p-kp}{(1-p)k}\theta_{k}=\frac{(1-p)k+p}{(1-p)k}\theta_k=\left(1+\frac{p}{(1-p)k} \right)\theta_k.
\end{align}
This proves the statement by induction. Inductively we obtain that 
\begin{align}
\theta_k\geq \prod_{i=1}^{k-1} \left(1+\frac{p}{(1-p)i} \right)\theta_1.
\end{align}
If $\theta_1>0$, then this product is strictly larger than the sum $\frac{p}{1-p}\theta_1\sum_{i=1}^{k}\frac{1}{k}$, which diverges as $k\to \infty$. This implies that $\theta_k\to \infty$, contradicting the boundedness of the sequence $(\theta_k)_{k=0}^{\infty}$. Hence $\theta_1=0$ and $d_1=\hat{d_1}$. But this implies that $d_k=\hat{d_k}$ for all $k\geq 1$. Hence there is a unique bounded sequence satisfying (\ref{def:degreeseq0}) and (\ref{def:degreeseq2}).
\end{proof}

We remark that although $d_0$ does not have any probabilistic interpretation, it will simplify the analysis later to have defined this. 

\section{Almost sure convergence of the degree distribution}
We now proceed to prove that we in fact have
\begin{align}
 \lim_{m\to \infty} \frac{D_{m,k}}{N_m}=d_k
\end{align}
almost surely for all $k\geq 1$. We will need the following auxiliary lemma.

\begin{lemma}\label{lem:lim}
Suppose we have a countable family $(Y_{m,k})_{k=1}^{\infty}$ of non--negative random variables indexed by discrete time $m$, with the property that $\sum_{k=1}^{\infty}Y_{m,k}=1$ for all $m$. Let $(b_k)_{k=1}^{\infty}$ be a non--negative sequence with $\sum_{k=1}^{\infty}b_k=1$ and such that
\begin{align}
  \liminf_{m\to \infty}Y_{m,k}\geq b_k
\end{align}
almost surely. Then \begin{align}
      \lim_{t\to \infty} Y_{m,k}=b_k
     \end{align}
exists almost surely.
\end{lemma}

\begin{proof}
The idea is to prove that $\limsup_{m\to \infty}Y_{m,k}\leq b_k$ for all $k\geq 1$. The following calculation is routine and only uses Fatou's lemma and well--known facts about the limit inferior and limit superior. For $k=1$ we have that
\begin{align}
\begin{split}
 \limsup_{m\to \infty}Y_{m,1} 
=\limsup_{m\to \infty}\left(1-\sum_{k=2}^{\infty}Y_{m,k} \right) 
&= 1+\limsup_{m\to \infty}\left(-\sum_{k=2}^{\infty}Y_{m,k}\right) \\
&= 1-\liminf_{m\to \infty}\sum_{k=2}^{\infty}Y_{m,k} \\
&\leq 1-\sum_{k=2}^{\infty}\liminf_{m\to \infty}Y_{m,k} \\
&\leq 1-\sum_{k=2}^{\infty}b_k \\
&=1-(1-b_1) \\
&=b_1.
\end{split}
\end{align}
A similar proof works for any $k\geq 1$, so we have that 
\begin{align}
\limsup_{m\to \infty}Y_{m,k}\leq b_k
\end{align}
almost surely for all $k\geq 1$. This implies that \begin{align}\lim_{m\to \infty} Y_{m,k}=b_k\end{align} almost surely.
\end{proof}

Aided by this lemma we can now show that $\lim_{m\to \infty}\frac{D_{m,k}}{N_m}=d_k$ almost surely. We note that the proof of this theorem relies on Lemma \ref{lem:probdist} which we prove later.
\begin{theorem}\label{thm:as}
For each $k\geq 1$ we have that \begin{align}
 \lim_{m\to \infty} \frac{D_{m,k}}{N_m}=d_k
\end{align}
almost surely, where $d_k$ is the unique bounded positive sequence satisfying
\begin{align}
d_0&=\frac{1-p}{p}, \label{def:degseq0}  \\
d_k&=\frac{pkd_{k-1}+(1-p)kd_{k+1}}{k+p}, \quad k\geq 1. \label{def:degseqk}
\end{align}
almost surely.
\end{theorem}

\begin{proof}
By Theorem \ref{thm:as:degree} the sequence $(d_k)_{k=1}^{\infty}$ is unique, bounded and positive. Furthermore, $\liminf_{m\to \infty}\frac{D_{m,k}}{N_m}\geq d_k$ holds almost surely. Since $\sum_{k=1}^{\infty}D_{m,k}=N_m$ for all $m\geq 1$ it holds that $\sum_{k=1}^{\infty}\frac{D_{m,k}}{N_m}=1$. In Section $4$ we will determine the sequence $(d_k)_{k=1}^{\infty}$ explicitly and show that it has the property $\sum_{k=1}^{\infty}d_k=1$ for all $0<p<1$. Invoking Lemma \ref{lem:lim} with $Y_{m,k}=\frac{D_{m,k}}{N_m}$ we find that 
\begin{align}\lim_{m \to \infty} \frac{D_{m,k}}{N_m}=d_k\end{align}
exists almost surely.
\end{proof}


\section{Analysis of the probability distribution}
In this section we analyse the sequence $(d_k)_{k=1}^{\infty}$ and show that such a sequence exists. We find exact expressions in terms of the hypergeometric function and later derive asymptotics. For the subcritical and supercritical cases we use Laplace's solution method of recursions with polynomial coefficients. The critical case is covered in \cite{BackhauszMori}, so we mention it only briefly below.

\subsection{Exact expressions of the asymptotic degree distribution}
\subsubsection{Supercritical case}
Suppose that $\frac{1}{2}<p<1$. Using Laplace's method we can determine the exact solution to the sequence $(d_k)_{k=0}^{\infty}$ defined in (\ref{def:degreeseq0}) and (\ref{def:degreeseq2}). For details about Laplace's method we refer the reader to \cite{Jordan, Norlund}. It is used in a similar vein in among others \cite{CooperFriezeVera, DeijfenLindholm, PralatWang, Vallier}. 
\begin{theorem}
Let $\frac{1}{2}<p<1$. Then
\begin{align}
 d_k=\gamma\int_0^1\frac{t^k (1-t)^{\beta-1}}{(1-\gamma t)^{\beta+1}}dt
\end{align}
where $\beta=\frac{p}{2p-1}$ and $\gamma=\frac{1-p}{p}$.
\label{thm:exact}
\end{theorem}

\begin{proof}
First we put (\ref{def:degseqk}) on the form \begin{equation}k(1-p)d_{k+1}-(k+p)d_{k}+kpd_{k-1}=0, \qquad k\geq 1,\end{equation} and and identify constants $\Phi_0, \Phi_1, \Phi_2, \Psi_0, \Psi_1$ and $\Psi_2$ so that it can be written on the form \begin{equation} \left(\Phi_2(k+1)+\Psi_2 \right)d_{k+1}+\left(\Phi_1k+\Psi_1 \right)d_{k}+\left(\Phi_0(k-1)+\Psi_0 \right)d_{k-1}=0. \end{equation} Comparing coefficients we see that 
\begin{equation}
\begin{cases} 
	\Phi_2 = (1-p) \\ \Phi_1 = -1 \\ \Phi_0 = p 
\end{cases} \end{equation} and 
\begin{equation}\begin{cases} 
	\Psi_2 = -(1-p) \\ \Psi_1 = -p \\ \Psi_0 = p. 
\end{cases} \end{equation}

Now define $\hat{\Phi}(t)=\Phi_2t^2+\Phi_1t+\Phi_0$ and $\hat{\Psi}(t)=\Psi_2t^2+\Psi_1t+\Psi_0$, that is
\begin{equation} \begin{cases} 
	\hat{\Phi}(t)=(1-p)t^2-t+p=(t-1)((1-p)t-p), \\ 
	\hat{\Psi}(t)=-(1-p)t^2-pt+p. 
\end{cases}\end{equation}

If we put \begin{equation}d_k=\int_{t_0}^{t_1}t^{k-1}v(t)dt, \qquad k\geq 0,\end{equation} we need only determine $t_0<t_1$ and a function $v:[t_0,t_1]\to \mathbb{R}$ such that
\begin{equation}\frac{v'(t)}{v(t)}=\frac{\hat{\Psi}(t)}{t\hat{\Phi}(t)} =\frac{-(1-p)t^2-pt+p}{t(t-1)((1-p)t-p)} \label{eq:der}
\end{equation} 
and \begin{equation}\left[t^k v(t)\hat{\Phi}(t) \right]_{t_0}^{t_1}=0 \label{eq:int}\end{equation} in order to have a solution to the recursion. Define $\gamma=\frac{1-p}{p}$ so that (\ref{eq:der}) can be written on the form
\begin{align}
 \frac{v'(t)}{v(t)}=\frac{\gamma t^2+t-1}{t(t-1)(1-\gamma t)}. \label{eq:diff}
\end{align}

Integrating we find that
\begin{align}
 \log v(t)=\frac{\gamma \log(1-t)+(1-\gamma)\log(t)+(\gamma -2 )\log(1-\gamma t)}{1-\gamma}+\const.
\end{align}
Exponentiating and defining $\beta=\frac{p}{2p-1}$ we finally obtain
\begin{align}
\begin{split}
 v(t)
&=Rt(1-t)^{-\frac{\gamma}{\gamma-1}}(1-\gamma t)^{\frac{1}{\gamma-1}-1} \\
&=Rt(1-t)^{\beta-1}(1-\gamma t)^{-(\beta+1)}
\end{split}
\end{align}
where $R$ is some positive constant, to be determined. We note that 
\begin{align}t^kv(t)\hat{\Phi}(t)=\const \cdot t^{k+1}(1-t)^{\beta}(1-\gamma t)^{-\beta}.\label{eq:tvphi}\end{align} 

For $\frac{1}{2}<p<1$ we have that $\beta>0$ and $\gamma<1$, so the function $t^kv(t)\hat{\Phi}(t)$ is continuous on $[0,1]$ and zero at the endpoints. Choosing $t_0=0$ and $t_1=1$ we ensure that (\ref{eq:int}) is satisfied. Note also that the singularities of the right hand side in (\ref{eq:diff}) occur at $t=0, t=1$ and $t=\gamma^{-1}>1$.

%
%

We thus obtain \begin{align} d_k
=\int_{0}^{1}t^{k-1}v(t)dt 
&=R\int_{0}^{1}\frac{t^{k}(1-t)^{\beta-1}}{(1-\gamma t)^{\beta+1}}dt.\label{eq:exact}
\end{align}

The integral in (\ref{eq:exact}) is a hypergeometric integral \cite{SpecialFunctions}, and can be expressed in terms of the hypergeometric function $_2F_1$. To be precise, we have that
\begin{align}
 d_k=R\frac{\Gamma(k+1)\Gamma(\beta)}{\Gamma(\beta+k+1)}\ _2F_1(\beta+1 , k+1,\beta + k+1; \gamma).
\end{align}

We now determine the constant $R$. Recall the Pochhammer symbol $(x)_k=\prod_{i=0}^{k-1}(x+i)$. Using a standard series expansion of the hypergeometric function, c.f. \cite{SpecialFunctions}, we find for $k=0$ that 
\begin{align}
\begin{split}
 d_0
&=R\frac{\Gamma(1)\Gamma(\beta)}{\Gamma(\beta+1)}\ _2F_1(\beta+1 , 1,\beta + 1; \gamma) \\
&=R\beta^{-1} \sum_{n=0}^{\infty}\frac{(\beta+1)_n(1)_n}{(\beta+1)_nn!}\gamma^n  \\
&=R\beta^{-1}\sum_{n=0}^{\infty}\gamma^n \\
&=R\beta^{-1} (1-\gamma)^{-1}\\
&=R.
\end{split}
\end{align}
Since $d_0=\gamma$ we find that $R=\gamma$. An alternative approach to evaluating $R$ is to compute the integral for $d_0$ directly.

Having determined $R$ we know that
\begin{align}
 d_k=\gamma\int_0^1\frac{t^k (1-t)^{\beta-1}}{(1-\gamma t)^{\beta+1}}dt
\end{align}
for all $k\geq 0$.
 \end{proof}

\subsubsection{Subcritical case}
\begin{theorem}
Let $0<p<\frac{1}{2}$. Then
\begin{align}
 d_k=\gamma^{-k}\int_0^1 \frac{t^k(1-t)^{-1-\beta}}{(1-\gamma^{-1}t)^{1-\beta}}dt
\end{align}
where $\beta=\frac{p}{2p-1}$ and $\gamma=\frac{1-p}{p}$.
\label{thm:subexact}
\end{theorem}
\begin{proof}
The subcritical case is similar to the supercritical case, so we give less detail here. Indeed, we can follow the analysis for the supercritical case up until (\ref{eq:tvphi}). When $0<p<\frac{1}{2}$ we have that $\beta<0$ and $\gamma>1$, so we need to make the choices $t_0=0$ and $t_1=\gamma^{-1}$ instead. This implies that 
\begin{align}
 d_k=R\int_0^{\gamma^{-1}}\frac{s^k(1-s)^{\beta-1}}{(1-\gamma s)^{\beta +1}}ds=R\gamma^{-k-1}\int_0^1 \frac{t^k(1-t)^{-1-\beta}}{(1-\gamma^{-1}t)^{1-\beta}}dt,
\end{align}
where we used the change of variables $t=\gamma s$. Again using the fact that $d_0=\gamma$ and the interpretation of this integral as the hypergeometric function we find that $R=\gamma$. Hence
\begin{align}
 d_k=\gamma^{-k}\int_0^1 \frac{t^k(1-t)^{-1-\beta}}{(1-\gamma^{-1}t)^{1-\beta}}dt.
\end{align}
\end{proof}

\subsubsection{Critical case} In the critical case $p=\frac{1}{2}$ we need to analyse the sequence defined by
\begin{align}
 &d_0=1,\label{def:exseq1} \\
 &d_k=\frac{kd_{k-1}+kd_{k+1}}{2k+1}, \qquad k\geq 1. \label{def:exseqk}
\end{align} 

As noted in the introduction, Backhausz and M\'ori introduced in \cite{BackhauszMori} a dynamic random graph model very similar to ours, with the difference that every other step was a duplication step and every other step a deletion step. For $p=\frac{1}{2}$ this is in some sense true on average in our model. Therefore it is not surprising that the sequence given by (\ref{def:exseq1}) and (\ref{def:exseqk}) appears as the asymptotic degree distribution in \cite{BackhauszMori} as well. Backhausz and M\'ori found the exact solution
\begin{align}
 d_k= k\int_0^{\infty}\frac{t^{k-1}e^{-t}}{(1+t)^{k+1}}dt.\label{eq:critexact}
\end{align}
and the asymptotic result 
\begin{align}d_k\sim (e\pi)^{1/2}k^{1/4}e^{-2\sqrt{k}}.\label{eq:critasymptotics}
\end{align}

This expression decreases slower than any exponential function and decreases faster than any polynomial, which demonstrates the fact that the asymptotic degree distribution in the critical case $p=\frac{1}{2}$ lies between a power--law and exponential decay.

It is worth noting that the integral in (\ref{eq:critexact}) appears as the limiting case as $p\to \frac{1}{2}$ in Theorem \ref{thm:exact} and Theorem \ref{thm:subexact}, that is
\begin{align}
\lim_{p\downarrow \frac{1}{2}}\gamma \int_0^1 \frac{t^k(1-t)^{\beta-1}}{(1-\gamma t)^{\beta+1}}dt = k\int_0^{\infty}\frac{t^{k-1}e^{-t}}{(1+t)^{k+1}}dt = \lim_{p\uparrow \frac{1}{2} }\gamma^{-k}\int_0^1 \frac{t^k(1-t)^{-1-\beta}}{(1-\gamma^{-1}t)^{1-\beta}}dt. \label{eq:p1/2}
\end{align}
This is justified by noting that the left--most and right--most integrands both can be dominated by an integrable function, so the dominated convergece theorem allows us to interchange the limit and integral sign. The integrands both converge pointwise to $\frac{t^k}{(1-t)^2}e^{-\frac{t}{1-t}}$ as $p\to \frac{1}{2}$. The middle integral then follows by a change of variables.


\subsection{Asymptotic results}
In this section we analyse the exact expressions and show asymptotic results as $k\to \infty$. We leave out the critical case and refer the reader to (\ref{eq:critasymptotics}).
\subsubsection{Supercritical case}
\begin{theorem}
For $\frac{1}{2}<p<1$, asymptotically as $k\to \infty$ the degree sequence $(d_k)_{k=1}^{\infty}$ satisfies
\begin{align}d_k \sim \gamma\beta^{\beta}\Gamma(\beta+1) k^{-\beta}
\label{eq:supasymptotics}\end{align}
where $\beta=\frac{p}{2p-1}$ and $\gamma=\frac{1-p}{p}$.
\label{thm:asymptotics}
\end{theorem}

\begin{proof}
Recall first the Beta integral
\begin{align}
\int_0^1 t^k(1-t)^{\beta-1}dt=\frac{\Gamma(k+1)\Gamma(\beta)}{\Gamma(k+\beta+1)}.
\end{align}

We bound the integral in Theorem \ref{thm:exact} from above and below. Estimating from above we find that

\begin{align}
 \gamma \int_0^1 \frac{t^k (1-t)^{\beta -1}}{(1-\gamma t)^{\beta +1}}dt 
& \leq \frac{\gamma}{(1-\gamma)^{\beta+1}} \int_0^1 t^k(1-t)^{\beta -1}dt \\
&=\frac{\gamma}{(1-\gamma)^{\beta+1}} \frac{\Gamma(k+1)\Gamma(\beta)}{\Gamma(k+\beta+1)} \\
&\sim  \frac{\gamma}{(1-\gamma)^{\beta+1}}\Gamma(\beta)k^{-\beta}.
\end{align}

In order to simplify notation, define 
\begin{align}
 A_k=\frac{\gamma}{\left(1-\gamma\left(1-1/\sqrt{k} \right) \right)^{\beta+1}}
\end{align}
and note that $A_k\to \frac{\gamma}{(1-\gamma)^{\beta+1}}$ as $k\to \infty$. We obtain the lower bound
\begin{align}
  \gamma\int_0^1 \frac{t^k (1-t)^{\beta -1}}{(1-\gamma t)^{\beta +1}}dt 
& \geq \gamma \int_{1-1/\sqrt{k}}^1\frac{t^k (1-t)^{\beta -1}}{(1-\gamma t)^{\beta +1}}dt \\
& \geq A_k\int_{1-1/\sqrt{k}}^1t^k(1-t)^{\beta-1}dt \\
& =A_k\left[\int_0^1 t^k(1-t)^{\beta-1}dt-\int_0^{1-1/\sqrt{k}}t^k(1-t)^{\beta-1}dt \right] \\
&=A_k\left[\frac{\Gamma(k+1)\Gamma(\beta)}{\Gamma (k+1+\beta)}-O\left(\int_0^{1-1/\sqrt{k}} t^kdt\right) \right] \\
&=A_k\left[\frac{\Gamma(k+1)\Gamma(\beta)}{\Gamma (k+1+\beta)}-O\left(e^{-\sqrt{k}} \right) \right] \\
&\sim \frac{\gamma}{(1-\gamma)^{\beta+1}}\Gamma(\beta)k^{-\beta}
\end{align}

The theorem is proved once we note that 
\begin{align}
 \frac{\gamma}{(1-\gamma)^{\beta+1}}\Gamma(\beta)k^{-\beta} = \gamma\beta^{\beta}\Gamma(\beta+1) k^{-\beta}.
\end{align}
\end{proof}

Alternatively, one could use a series expansion of the $\frac{1}{(1-\gamma t)^{\beta+1}}$--term and then apply Stirling's formula to achieve the same result. Estimations like these were done in \cite{CooperFriezeVera}.

Since $\beta=\frac{p}{2p-1}\in (1,\infty)$ when $\frac{1}{2}<p<1$, the supercritical case gives rise to a power--law with exponent greater than $1$.

\subsubsection{Subcritical case}
The proof of the following theorem is similar to Theorem \ref{thm:asymptotics}, and the proof is omitted.
\begin{theorem}\label{thm:subasymptotics}
 For $0<p<\frac{1}{2}$, asymptotically as $k\to \infty$ we have that 
\begin{align}
d_k\sim (-\beta)^{-1}(1-\beta)^{1-\beta}\Gamma(1-\beta) \gamma^{-k}k^{\beta} \label{eq:subasymptotics}
\end{align}
where $\beta=\frac{p}{2p-1}$ and $\gamma=\frac{1-p}{p}$.
\end{theorem}

\begin{remark}
  If $0<p<\frac{1}{2}$ we have that $\gamma>1$, so the factor of $\gamma^{-k}$ forces most vertices to have very low degree. This is what one would expect, since on average there are more deletion steps (which reduce the size of a clique and create new isolated vertices) than duplication steps (which increase the size of a clique). 
\end{remark}

\subsection{Proof that $(d_k)_{k=1}^{\infty}$ defines a probability distribution}

Finally we prove the sequence $(d_k)_{k=1}^{\infty}$ defines a probability distribution. This is the final step in the proof of Theorem \ref{thm:as}.

\begin{lemma}
 Let $0<p<1$. The unique bounded sequence $(d_k)_{k=1}^{\infty}$ defined by
\begin{align}
d_0&=\frac{1-p}{p},  \label{def:prob0}  \\
d_k&=\frac{pkd_{k-1}+(1-p)kd_{k+1}}{k+p}, \quad k\geq 1 \label{def:probk}
\end{align} defines a probability distribution, i.e. \begin{align} \sum_{k=1}^{\infty}d_k=1.\end{align}
\label{lem:probdist}
\end{lemma}

\begin{proof}
Summing over $k=1,\dots n$ we find that
\begin{align}
\begin{split}
 \sum_{k=1}^{n}(k+p)d_k
&=(1-p)(1+d_2)+\sum_{k=2}^n k(pd_{k-1}+(1-p)d_{k+1}) \\
&=(1-p)(1+d_2)+p\sum_{k=1}^n(k+1)d_k+(1-p)\sum_{k=1}^n (k-1)d_k \\
& \qquad - (1-p)d_2-p(n+1)d_n+(1-p)nd_{n+1} \\
&=(1-p)+\sum_{k=1}^nkd_k+(2p-1)\sum_{k=1}^nd_k-p(n+1)d_n+(1-p)nd_{n+1} 
\end{split}
\end{align}
Hence
\begin{align}
 \sum_{k=1}^nd_k = 1 +\frac{1}{1-p}\left(-p(n+1)d_n+(1-p)nd_{n+1} \right).\label{eq:sum=1}
\end{align}
The asymptotics in (\ref{eq:critasymptotics}), (\ref{eq:supasymptotics}) and (\ref{eq:subasymptotics}) imply that $d_n=o(n^{-1})$ for any $0<p<1$. Hence the sum in (\ref{eq:sum=1}) tends to $1$ as $n \to \infty$. 
\end{proof}

\begin{remark}
The above proof shows that any sequence satisfying (\ref{def:prob0}) and (\ref{def:probk}) will sum to $1$ under the assumption that $d_n=o(n^{-1})$. Alternatively one could prove this directly from the exact expressions. For instance, in the supercritical case, it is not difficult to show that
\begin{align}
 \sum_{k=1}^{\infty}d_k=\sum_{k=0}^{\infty}d_k-d_0=\sum_{k=0}^{\infty}\gamma\int_0^1\frac{t^k(1-t)^{\beta-1}}{(1-\gamma t)^{\beta+1}}dt-\gamma=1
\end{align}
by interchanging the summation sign and the integral sign, and using the geometric series.
\end{remark}

\section*{Acknowledgements} The author thanks Professor Svante Janson for suggesting the topic and for providing plenty of helpful ideas and suggestions to improve the present manuscript.

\bibliographystyle{plain}
\bibliography{arxiv}

\begin{thebibliography}{10}

\bibitem{SpecialFunctions}
G.~E. Andrews, R.~Askey, and R.~Roy.
\newblock {\em Special Functions}.
\newblock Cambridge University Press, 1999.

\bibitem{BackhauszMori}
{\'A}.~Backhausz and T.~M{\'o}ri.
\newblock Asymptotic properties of a random graph with duplications.
\newblock {\tt http://arXiv:1308.1506v2}, 2013.
\newblock Preprint.

\bibitem{BackhauszMoriLemma}
{\'A}.~Backhausz and T.~M{\'o}ri.
\newblock A random model of publication activity.
\newblock {\em Discrete Appl. Math.}, 162:78--89, 2014.

\bibitem{BarabasiAlbert}
A.-L. Barab{\'a}si and R.~Albert.
\newblock Emergence of scaling in random networks.
\newblock {\em Science}, 286(5439):509--512, 1999.

\bibitem{BollobasRiordanScaleFree}
B.~Bollob{\'a}s, O.~Riordan, J.~Spencer, and G.~Tusn{\'a}dy.
\newblock The degree sequence of a scale-free random graph process.
\newblock {\em Random Structures and Algorithms}, 18:279--290, 2001.

\bibitem{ChampagnatLambert2012}
N.~Champagnat and A.~Lambert.
\newblock Splitting trees with neutral {P}oissonian mutations i: Small
  families.
\newblock {\em Stochastic Processes and their Applications}, 122(3):1003--1033,
  2012.

\bibitem{ChampagnatLambert2012b}
N.~Champagnat, A.~Lambert, and M.~Richard.
\newblock Birth and death processes with neutral mutations.
\newblock {\em International Journal of Stochastic Analysis}, 2012:20 pages,
  2012.

\bibitem{ChungLu2004}
F.~Chung and L.~Lu.
\newblock Coupling online and offline analyses for random power law graphs.
\newblock {\em Internet Math.}, 1(4):409--461, 2004.

\bibitem{CooperFriezeVera}
A.~Cooper, A.~Frieze, and J.~Vera.
\newblock Random deletion in a scale free random graph process.
\newblock {\em Internet Math.}, 124(4):463--483, 2004.

\bibitem{DeijfenLindholm}
M.~Deijfen and M.~Lindholm.
\newblock Growing networks with preferential deletion and addition of edges.
\newblock {\em Physica A: Statistical Mechanics and its Applications},
  388(19):4297--4303, 2009.

\bibitem{Jordan}
C.~Jordan.
\newblock {\em Calculus of Finite Differences}.
\newblock Rottig and Romwalter, Budapest, 1939.

\bibitem{MoriRandomTrees}
T.~M{\'o}ri.
\newblock On random trees.
\newblock {\em Studia Sci. Math. Hungar.}, 39(1-2):143--155, 2002.

\bibitem{Norlund}
N.-E. N{\"o}rlund.
\newblock {\em Differenzenrechnung}.
\newblock Springer Verlag, 1924.

\bibitem{pakes1989}
A.~G. Pakes.
\newblock An infinite alleles version of the markov branching process.
\newblock {\em J. Australian. Math. Soc.}, 46:146--170, 1989.

\bibitem{PralatWang}
P.~Pralat and C.~Wang.
\newblock An edge deletion model for complex networks.
\newblock {\em Theoretical Computer Science}, 412(39):5111--5120, 2011.

\bibitem{Vallier}
T.~Vallier.
\newblock Transition of the degree sequence in the random graph model of
  {C}ooper, {F}rieze and {V}era.
\newblock {\em Stochastic Models}, 29(3):341--352, 2013.

\bibitem{WuDongLiuCai}
X.-Y. Wu, Z.~Dong, K.~Liu, and K.-Y. Cai.
\newblock On the degree sequence of an evolving random graph process and its
  critical phenomenon.
\newblock {\em Journal of Applied Probability}, 46(4):1213--1220, December
  2009.

\end{thebibliography}

\end{document}